\documentclass[12pt, reqno]{amsart}

\newtheorem{theorem}{Theorem}[section]

\theoremstyle{definition}

\theoremstyle{remark}
\newtheorem{remark}[theorem]{Remark}
\numberwithin{equation}{section}

\usepackage{graphicx}
\usepackage{psfrag}
\usepackage{epsfig}

\begin{document}
\title[Partially Gaussian Stationary Stochastic Processes in Discrete Time]{Partially Gaussian Stationary Stochastic Processes in Discrete Time}

\author{K. R. Parthasarathy}
\address{Theoretical Statistics \& Mathematics Unit, Indian Statistical Institute, Delhi Centre, 7, S. J. S. Sansanwal Marg, New Delhi - 110016, India}
\email{krp@isid.ac.in}




 \begin{abstract}
We present here an elementary example, for every fixed positive integer $k,$ of a strictly stationary nongaussian stochastic process in discrete time, all of whose $k$-marginals are gaussian.

 \end{abstract}

\maketitle

\section{Introduction}
It is well known that, for every positive integer $n>1,$ there exists a probability distribution in $\mathbb{R}^n$ which is not gaussian but has all its $(n-1)$-dimensional marginal distributions gaussian. (See, for example, Section 10.3 in Stoyanov \cite{js}.) Using the finer theory of pathwise stochastic integrals and martingale methods, F\"oller, Wu and Yor \cite{hf} have shown that, for every positive integer $k,$ there exists a variety of nongaussian stochastic processes with continuous trajectories in the interval $[0,\infty)$ which have the same $k$-dimensional marginals as the standard brownian motion process. Here we present an elementary example, for every fixed positive integer $k > 1,$ a discrete time stationary stochastic process which is not gaussian but has all its $(k-1)$-marginals gaussian. However, we do not know how to construct such processes in continuous time.

\section{The basic construction}

Let $k > 1$ be any fixed positive integer and let $\nu$ be a probability distribution in $\mathbb{R}^k,$ which is not gaussian but has all its $(k-1)$-marginals gaussian with mean $\mathbf{0}$ and covariance matrix identity. For example, we may choose $\nu$ to have the probability density function
$$\psi (\mathbf{x}) = (2 \pi)^{-\frac{k}{2}} \left \{ 1 + x_1 x_2 \cdots x_k \quad e ^{-\frac{1}{2} |\mathbf{x}|^{2}}  \right \} e ^{-\frac{1}{2} |\mathbf{x}|^{2}}$$
where $\mathbf{x} = ( x_1, x_2, \ldots, x_k).$ If $(X_1, X_2, \ldots, X_k)$ is an $\mathbb{R}^k$-valued random variable with distribution $\nu,$  then the sequence $X_1, X_2, \ldots, \widehat{X}_i, \ldots, X_k$ with the $i$-th term omitted consists of i.i.d. $N(0,1)$ random variables, for each $i.$

Now consider a bilateral sequence $\{(X_{n1}, X_{n2}, \ldots  X_{nk}), - \infty < n < \infty\}$ of i.i.d $\mathbb{R}^k$-valued random variables with the common distribution $\nu$ as described in the preceding paragraph. Define
$$Y_n = X_{nk} + X_{n+1 \,\, k-1} + X_{n+2 \,\,k-2} + \cdots + X_{n+k-1\,\,1}, - \infty < n < \infty. $$
It is to be noted that the sum of the two suffixes in each summand on the right hand side is equal to $n+k.$

\begin{theorem}\label{thm2.1}
The sequence $\{Y_n, - \infty < n < \infty \}$ is strictly stationary, $(k-1)$-step independent with every $(k-1)$-dimensional marginal being gaussian with mean $\mathbf{0}$ and covariance matrix $kI,$ $I$ being the identity matrix of order $k-1.$ In particular, $\{Y_n\}$ is ergodic.  
\end{theorem}

\begin{proof}
 Fix an integer $m$ and consider the two sets $\{Y_n, n \leq m\}$ and  $\{Y_n, n \geq m+k\}.$ Since $Y_m = X_{mk} + X_{m+1 \,\,k-1} +\cdots+ X_{m+k-1 \,\,1}$ and $Y_{m+k}  = Y_{m+k \,\, k} +  X_{m+k+1 \,\,k-1} + \cdots +  X_{m+2 \,k-1 \,\,1}$ and the first suffix in the last summand in the definition of $Y_m$ is less than the first suffix in the first summand in the definition of $Y_{m+k}$ it follows that the two sets $\{ Y_n, n \leq m \}$ and $\{Y_n, n \ge m+k \}$ are independent. In other words $\{Y_n\}$ is a $(k-1)$-step independent process.

We now look at the column vector-valued random variable
$$\left [  \begin{array}{c} Y_{n+1} \\  Y_{n+2} \\ \vdots \\ Y_{n+m}  \end{array} \right ] = \left [  \begin{array}{l} X_{n+1\,\,k} + X_{n+2 \,\,k-1} + \cdots + X_{n+k\,\,1 }  \\   X_{n+2\,\,k} + X_{n+3 \,\,k-1} + \cdots + X_{n+k+1\,\,1 }  \\    \\ X_{n+m\,\,k} + X_{n+m+1 \,\,k-1} + \cdots + X_{n+m+k-1\,\,1 }   \end{array} \right ] $$
and express it as $S_1 + S_2 + S_3$ where
\begin{eqnarray*}
S_1 &=& \left [ \begin{array}{c} X_{n+1\,\,k} \\ 0 \\ \vdots \\ 0  \end{array} \right ] + \left [ \begin{array}{c} X_{n+2 \,\,k-1} \\ X_{n+2 \,\,k} \\ 0 \\ \vdots \\ 0 \end{array} \right ] + \cdots  + \left [ \begin{array}{c} X_{n+k-1\,\,2} \\  X_{n+k-1\,\,3} \\ \vdots \\ X_{n+k-1\,\,k} \\ 0 \\ \vdots \\ 0 \end{array} \right ],\\ 
S_2 &=& \left [ \begin{array}{c} X_{n+k\,\,1} \\ X_{n+k\,\,2}\\ \vdots \\  X_{n+k\,\,k} \\ 0 \\ \vdots \\ 0 \end{array} \right ] + \left [ \begin{array}{c} 0 \\ X_{n+k+1 \,\,1} \\ X_{n+k+1 \,\,2} \\ \vdots \\ X_{n+k+1 \,\,k}  \\ 0 \\ \vdots \\ 0 \end{array} \right ] + \cdots  + \left [ \begin{array}{c} 0 \\ \vdots \\ 0 \\ X_{n+m\,\,1} \\  X_{n+m\,\,2} \\ \vdots \\ X_{n+m\,\,k} \end{array} \right ],\\
S_3 &=& \left [ \begin{array}{c} 0 \\ \vdots \\ 0 \\ X_{n+m+1\,\,1} \\  X_{n+m+1\,\,2} \\ \vdots \\ X_{n+m+1\,\,k-1}  \end{array} \right ] + \left [ \begin{array}{c} 0 \\ \vdots \\ 0 \\ X_{n+m+2\,\,1} \\  X_{n+m+2\,\,2} \\ \vdots \\ X_{n+m+2\,\,k-2} \end{array} \right ] + \cdots  + \left [ \begin{array}{c} 0 \\ \vdots \\ \vdots \\ 0 \\ X_{n+m+k-1\,\,1} \end{array} \right ].                                
\end{eqnarray*}
In each column on the right hand side of $S_1$ or $S_3$ there are at most $k-1$ nonzero entries whereas in each column on the right hand side of $S_2$ there are exactly $k$ entries. All the column vectors appearing in $S_1, S_2, S_3$ together are mutually independent. By the choice of the measure $\nu,$ $S_1$ and $S_3$ are gaussian random vectors. Denote by $\nu([i,j])$ the $(j-i+1)$-dimensional standard normal distribution imbedded in $\mathbb{R}^m$ so that the first $i-1$ and the last $m-j$ coordinates are $0$ when $1 \leq i \leq j \leq m.$ Similarly, denote by $\nu ([j, k+j-1])$ the $k$-dimensional distribution $\nu$ imbedded in $\mathbb{R}^m$ with the first $j-1$ and the last $m-k-j+1$ coordinates $0$ for $1 \leq j \leq m-k+1,$ assuming $m \ge k.$ Then it follows that the random variables $Y_{n+1}, Y_{n+2}, \ldots, Y_{n+m}$ expressed as a single column vector has the $m$-dimensional distribution $\nu_m$ (in $\mathbb{R}^m$) given by
\begin{eqnarray*}
\nu_m &=& \mu ([1,1]) \ast \mu([1,2]) \ast \ldots \ast \mu ([1, k-1]) \\
&& \ast \nu ([1,k]) \ast \nu ([2,k+1]) \ast \ldots \ast \nu ([m-k+1, m) \\
&& \ast \mu ([m-k+2, m]) \ast \mu ([m-k+3, m]) \ast \ldots \ast \mu ([m,m]),
\end{eqnarray*}
for every $m \ge k.$ Since $\nu_m$ is independent of $n$ it follows that $\{Y_n\}$ is a strictly stationary process. Since $\nu([1,k])$ is nongaussian it is clear that $\nu_m$ is not gaussian for every $m \ge k.$

We now observe that $Y_n,$ being a sum of $k$ independent $N(0,1)$ random variables, is an $N(0,k)$ variable with mean $0$ and variance $k.$ Now consider the pair $(Y_0, Y_m).$ If $m \ge k$ we have already seen that $Y_0$ and $Y_m$ are independent. If $m<k,$ we write
\begin{eqnarray*}
\left [\begin{array}{c}  Y_0 \\ Y_m \end{array} \right ] &=& \left [\begin{array}{c} X_{0k} + X_{1 \,\,k-1} + \cdots + X_{m-1 \,\, k-m+1} \\ 0 \end{array} \right ] +  \left [\begin{array}{c} X_{m\,\, k-m} \\ X_{m \,\,k} \end{array} \right ] \\
&&  +  \left [\begin{array}{c} X_{m+1 \,\, k-m-1} \\ X_{m+1 \,\,k-1}  \end{array} \right ] + \cdots + +  \left [\begin{array}{c} X_{k-1 \,\, 1} \\ X_{k-1 \,\,m+1} \end{array} \right ]\\
&&  +  \left [\begin{array}{c} 0 \\ X_{km} + X_{k+1 \,\,m-2} + \cdots + X_{k+m-1\,\,1}  \end{array} \right ].
\end{eqnarray*}
Now the special choice of $\nu$ implies that $Y_0$ and $Y_m$ are independent $N(0,k)$ random variables. Stationarity of the process $\{Y_n\}$ implies that $Y_{n_{1}}$ and $Y_{n_{2}}$ are independent $N(0,k)$ random variables for any $n_1, n_2.$

Now consider, for any $n_1 < n_2 < \cdots < n_{k-1}$ the random vector
$$\widetilde{\mathbf{Y}} = \left [ \begin{array}{c} Y_{n_{1}} \\  Y_{n_{2}} \\ \vdots \\ Y_{n_{k-1}} \end{array} \right ] = \left [ \begin{array}{c} X_{n_{1}k} + X_{n_{1} +1 \,\,k-1} + \cdots + X_{n_{1} + k-1\,\,1} \\ X_{n_{2}k} + X_{n_{2} +1 \,\,k-1} + \cdots + X_{n_{2} + k-1\,\,1} \\ \vdots \\  X_{n_{k-1} k} + X_{n_{k-1}+1 \,\,k-1} + \cdots + X_{n_{k-1} + k-1 \,\,1}  \end{array} \right ].$$ 
The right hand side can be expressed as a sum of column vectors in which the entries in each column are either $0$ or an $X_{rs}$ where the first suffix $r$ is fixed and the second suffix takes at most $k-1$ values from the set $\{1,2,\ldots,k\}.$ The different column vectors are independent and by the choice of $\nu$ each column has a multivariate gaussian distribution. Thus $\widetilde{\mathbf{Y}}$ is gaussian. Since any two $Y_i$ and $Y_j$ are independent where $k > 2,$ it follows that $Y_{n_{1}}, Y_{n_{2}}, \ldots, Y_{n_{k-1}}$ are i.i.d $N(0,k)$ random variables. This completes the proof.
\end{proof}
\vskip0.2in
\begin{remark}
 From the proof of Theorem \ref{thm2.1} it is clear that any $k-1$ of the random variables $\{Y_n\}$ are i.i.d $N(0,k).$ This motivates the introduction of the following notion of limited exchangeability. We say that a stationary random process $\{Z_n, - \infty < n < \infty\}$ is  $k$-{\it exchangeable} if any $Z_{n_{1}},Z_{n_{2}}, \ldots,Z_{n_{k}}$ has the same distribution for any $k$-point set $\{ n_1, n_2, \ldots, n_k\} \subset \mathbb{Z}.$ The probability measures of all such $k$-exchangeable stationary processes constitute a convex set. One wonders what are the extreme points of this convex set.
\end{remark}

\end{document}